\documentclass[letterpaper, 10 pt, conference]{ieeeconf}  

\IEEEoverridecommandlockouts                              

\usepackage{times}
\usepackage{enumerate}
\usepackage{mathtools}
\usepackage{dsfont}
\usepackage{amsmath}
\usepackage{amssymb} 
\usepackage{color}
\usepackage{multirow}
\newtheorem{theorem}{Theorem}

\newtheorem{problem}{Problem}
\newtheorem{corollary}{Corollary}
\DeclareMathOperator*{\argmin}{arg\,min}
\title{\LARGE \bf
Chance-Constrained Stochastic Optimal Control via\\ Path Integral and Finite Difference Methods}

\author{Apurva Patil$^1$ \and Alfredo Duarte$^2$ \and Aislinn Smith$^3$ \and Takashi Tanaka$^2$ \and Fabrizio Bisetti$^2$ 
\thanks{$^{1}$Walker Department of Mechanical Engineering, University of Texas at Austin. {\tt\small apurvapatil@utexas.edu}.
$^{2}$Department of Aerospace Engineering and Engineering Mechanics, University of Texas at Austin. {\tt\small aduarteg@utexas.edu}, {\tt\small ttanaka@utexas.edu}, {\tt\small fbisetti@utexas.edu}.
$^{3}$Department of Mathematics, University of Texas at Austin. {\tt\small aislinnsmith@utexas.edu}.}
}
\begin{document}

\maketitle
\thispagestyle{empty}
\pagestyle{empty}

\begin{abstract}
This paper addresses a continuous-time continuous-space chance-constrained stochastic optimal control (SOC) problem via a Hamilton-Jacobi-Bellman (HJB) partial differential equation (PDE). Through Lagrangian relaxation, we convert the chance-constrained (risk-constrained) SOC problem to a risk-minimizing SOC problem, the cost function of which possesses the time-additive Bellman structure. We show that the risk-minimizing control synthesis is equivalent to solving an HJB PDE whose boundary condition can be tuned appropriately to achieve a desired level of safety. Furthermore, it is shown that the proposed risk-minimizing control problem can be viewed as a generalization of the problem of estimating the risk associated with a given control policy. Two numerical techniques are explored, namely the path integral and the finite difference method (FDM), to solve a class of risk-minimizing SOC problems whose associated HJB equation is linearizable via the Cole-Hopf transformation. Using a 2D robot navigation example, we validate the proposed control synthesis framework and compare the solutions obtained using path integral and FDM.
\end{abstract}
\section{Introduction}\label{Sec: Introduction}
Safety-critical systems are required to be controlled within their predefined safe regions over the entire operation horizon. For the controller synthesis of such systems, one must balance the trade-off between control cost and safety. For instance, in motion planning, an autonomous vehicle is required to reach a goal position with minimum fuel or time consumption, while avoiding obstacles on the way. When deployed in real-life environments, these systems are subject to uncertainties that arise due to modeling errors, uncertain localization, and disturbances. While some studies consider robust control under set-bounded uncertainties \cite{kothare1996robust}, \cite{goulart2006optimization}, \cite{kuwata2007robust}, in many cases, modeling uncertainties with unbounded (e.g.  Gaussian) distributions is advantageous over a set-bounded approach \cite{blackmore2011chance}. In the case of unbounded uncertainties, it is generally difficult to guarantee safety against all realizations of noise. An alternative approach is to design a controller subject to the constraint that the probability of failure to satisfy given state constraints is bounded by a user-specified threshold. This problem is known as the \textit{chance-constrained} or \textit{risk-constrained} stochastic optimal control (SOC) problem \cite{blackmore2011chance}, \cite{oguri2019convex}.\par
Unfortunately, the probability of failure is in general challenging to evaluate and optimize against since it involves multidimensional integration \cite{paulson2019efficient}, \cite{frey2020collision}. Various approaches have been proposed in the literature that use conservative approximations (upper bounds) of the failure probability to provide tractable approximations of the risk-constrained SOC problem; for instance, approaches based on Boole's inequality \cite{blackmore2011chance}, \cite{ono2010chance}, Chebyshev inequality \cite{nakka2019trajectory}, moment-based surrogate \cite{paulson2019efficient}, and stochastic control barrier functions (CBFs) \cite{prajna2007framework}, \cite{yaghoubi2020risk}, \cite{santoyo2021barrier}. However, in these bounding-based approaches, the bounds can be too conservative, leading to overly cautious policies that compromise performance or can cause artificial infeasibilities \cite{chern2021safe}, \cite{patil2021upper}. In some cases, these bounds can also be lax and generate risky policies \cite{ariu2017chance}, \cite{frey2020collision}. A dynamic-programming-based solution approach for discrete-time chance-constrained SOC problem is proposed in \cite{ono2015chance} that uses Lagrangian relaxation to convert the risk-constrained problem to a risk-minimizing one, similar to this work. However, in order to restore the time-additive Bellman structure of the Lagrangian (required in dynamic programming), \cite{ono2015chance} uses a conservative approximation of the failure probability using Boole's inequality. In this work, we leverage the notion of exit time from the continuous-time stochastic calculus to formulate the risk-constrained SOC problem. The resulting Lagrangian of the risk-minimizing SOC problem possesses the time-additive structure, which allows us to rewrite the risk-minimizing synthesis problem as a problem of solving a Hamilton-Jacobi-Bellman (HJB) partial differential equation (PDE). This in turn allows us to evaluate and optimize against the failure probability without having to introduce any conservative approximations. \par 
The contributions of this work are summarized as follows: 
\begin{enumerate}
    \item We consider a continuous-time continuous-space risk-constrained SOC problem with control-affine state dynamics. Through Lagrangian relaxation, we convert the risk-constrained SOC problem to a risk-minimizing SOC problem whose cost function possesses the time-additive Bellman structure. We show that an optimal risk-minimizing policy can be synthesized by solving an HJB PDE whose boundary condition can be tuned appropriately to achieve a desired level of safety. 
    \item It is shown that the proposed risk-minimizing control problem can be viewed as a generalization of the risk estimation problem. The risk associated with any given control policy can be computed by solving a PDE which is a special case of the HJB PDE associated with the risk-minimizing SOC problem.
    \item We apply two numerical techniques, path integral and finite difference method (FDM), to solve a class of risk-minimizing SOC problems whose associated HJB equation is linearizable via the Cole-Hopf transformation. The proposed control synthesis framework is validated through a 2D robot navigation problem, and the solutions obtained using path integral and FDM are compared. 
\end{enumerate}
\subsection*{Notation}
Let $\mathbb{R}$, $\mathbb{R}^n$, and $\mathbb{R}^{m\times n}$ be the set of real numbers, the set of $n$-dimensional real vectors, and the set of $m\times n$ real matrices. 
Bold symbols such as $\pmb{x}$ represent random variables. Let $\mathds{1}_{\mathcal{E}}$ be an indicator function, which returns $1$ when the condition $\mathcal{E}$ holds and 0 otherwise. 
The $\bigvee$ symbol represents a logical OR implying existence of a satisfying event among a (possibly uncountable) collection. $\partial_x$ and $\partial^2_x$ are used to define, respectively, the first and second-order partial derivatives w.r.t. $x$. If $x$ is a vector then $\partial_x$ returns a column vector and $\partial^2_x$ returns a matrix. $\text{Tr}(A)$ denotes the trace of a matrix $A$.   
\section{Problem Formulation}
We consider a random process $\pmb{x}(t)\in\mathbb{R}^n$ driven by the following control-affine stochastic differential equation (SDE):
\begin{equation}\label{SDE}
    d\pmb{x}\!=\!\!{f}\!\left(\pmb{x}(t),\! t\right)\!dt\!+\!{g}\!\left(\pmb{x}(t), \!t\right)\!{u}(\pmb{x}(t),\! t)dt\!+\!{\sigma}\!\left(\pmb{x}(t),\! t\right)\!d\pmb{w},
\end{equation}
where ${u}\left(\pmb{x}(t), t\right)\in\mathbb{R}^m$ is a control input, $\pmb{w}(t)\in\mathbb{R}^k$ is a $k$-dimensional standard Wiener process on a suitable probability space $\left(\Omega, \mathcal{F}, P\right)$, ${f}\left(\pmb{x}(t), t\right)\in\mathbb{R}^n$, ${g}\left(\pmb{x}(t), t\right)\in\mathbb{R}^{n\times m}$ and ${\sigma}\left(\pmb{x}(t), t\right)\in\mathbb{R}^{n\times k}$. Throughout this paper, we assume sufficient regularity in the coefficients of (\ref{SDE}) so that there exists a unique strong solution to the SDE. At the initial time $t_0$, the state of the system $\pmb{x}(t_0)=x_0$ is fixed. Let $\mathcal{X}_{s}\subseteq\mathbb{R}^n$ be a bounded open set representing a safe region, $\partial\mathcal{X}_{s}$ be its boundary, and $\overline{\mathcal{X}_{s}}=\mathcal{X}_{s}\cup\partial\mathcal{X}_{s}$. For a given finite time horizon $t\in[t_0, T]$, let $\mathcal{Q}=\mathcal{X}_s\times[t_0, T)$ be a bounded open region, $\partial\mathcal{Q}=\left(\partial\mathcal{X}_s\times[t_0,T]\right)\cup\left(\mathcal{X}_s\times\{T\}\right)$ be its boundary, and $\overline{\mathcal{Q}}=\mathcal{Q}\cup\partial\mathcal{Q}=\overline{\mathcal{X}_s}\times[t_0, T]$. Define the \textit{exit time} $\pmb{t}_f$ as follows:
\begin{equation}\label{tf}
\pmb{t}_f \coloneqq 
\begin{cases}
T, & \!\!\!\!\!\!\!\!\!\!\!\!\!\!\!\text{if}\;\; \pmb{x}(t)\in\mathcal{X}_{s}, \forall t\in[t_0, T],\\
\text{inf}\;\{t\in[t_0, T] : \pmb{x}(t)\notin\mathcal{X}_{s}\}, & \text{otherwise}.
 \end{cases}
\end{equation}
Note that $\left(\pmb{x}(\pmb{t}_f),\pmb{t}_f\right)\in\partial{\mathcal{Q}}$. If the system (\ref{SDE}) leaves the region $\mathcal{X}_{s}$ at any time $t\in[t_0, T]$, then we say that it fails. The probability of failure $P_{fail}$ of system (\ref{SDE}) is defined as:   
\begin{equation}\label{pfail}
    P_{fail}=P\left(\bigvee_{t\in[t_0, T]} \pmb{x}(t)\notin \mathcal{X}_{s}\right).
\end{equation}
We suppose that the cost function of the optimal control problem is quadratic in the control input and has the form:
\begin{equation}\label{C}
\begin{split}
 C\left(x_0, t_0, {u}(\cdot)\right)&\coloneqq\mathbb{E}_{x_0,t_0}\Bigg[\psi\left(\pmb{x}(\pmb{t}_f)\right)\cdot\mathds{1} _{\pmb{x}(\pmb{t}_f)\in \mathcal{X}_{s}}+\\
 &\!\!\!\!\!\!\!\!\!\!\!\!\!\!\!\!\!\!\!\!\!\!\!\!\!\!\!\!\!\!\!\!\!\!\!\!\!\!\!\int_{t_0}^{\pmb{t}_f}\left(\frac{1}{2}\pmb{u}^T{R}\left(\pmb{x}(t), t\right)\pmb{u}+ V\left(\pmb{x}(t), t\right)\right)dt\Bigg]
\end{split}
\end{equation}
where $\psi\left(\pmb{x}(\pmb{t}_f)\right)$ denotes a terminal cost, $V\left(\pmb{x}(t), t\right)$ a state dependent running cost, and $R\left(\pmb{x}(t), t\right)\in\mathbb{R}^{m\times m}$ a given positive definite matrix (for all values of $\pmb{x}(t)$ and $t$). $\mathbb{E}_{x_0, t_0}\left[\cdot\right]$ denotes the expectation taken over trajectories of system (\ref{SDE}) starting at $(x_0, t_0)$ under the probability law $P$. In the rest of the paper, for notational compactness, we will drop the functional dependencies on $x$ and $t$ whenever it is unambiguous. \par
Now, we wish to find an optimal control policy such that $C\left(x_0, t_0, u(\cdot)\right)$ is minimal and the probability of failure is below a specified threshold. This problem can be formulated as a risk-constrained SOC problem as follows:
\begin{problem}[Risk-constrained SOC problem]\label{Problem: Risk-constrained SOC problem}
    \begin{equation}\label{CC-SOC}
    \begin{aligned}
    \argmin_{{u}(\cdot)}\;& \mathbb{E}_{x_0, t_0}\!\!\left[\!\psi\!\left(\pmb{x}(\pmb{t}_f\!)\right)\!\cdot\!\mathds{1} _{\pmb{x}(\pmb{t}_f)\in \mathcal{X}_{s}}\!\!+\!\!\!\int_{t_0}^{\pmb{t}_f}\!\!\!\left(\!\frac{1}{2}\pmb{u}^T\!\!{R}\pmb{u}\!+\!V\!\right)\!dt\!\right]\\
    \textrm{s.t.} \;\; & d\pmb{x}={f}dt+{g}\,{u}dt+{\sigma}d\pmb{w},\quad \pmb{x}(t_0)=x_0,\\
      &P\left(\bigvee_{t\in[t_0, T]} \pmb{x}(t)\notin \mathcal{X}_{s}\right)\leq\Delta,
    \end{aligned}    
    \end{equation}
    where $\Delta\in(0,1)$ represents a given risk tolerance over the horizon $[t_0, T]$, and the admissible policy $u(\cdot)$ is measurable with respect to the $\sigma$-algebra adapted to $\pmb{x}(t)$.
\end{problem}
Note that our cost function is defined over the time horizon $[t_0, \pmb{t}_f]$ instead of $[t_0, T]$, because we assume that once the system fails, there is no need to consider the cost incurred after that. Moreover, this formulation will be useful to construct a time-additive Lagrangian for the risk-minimizing SOC problem as shown in the next section.
\section{Synthesis of a Risk-Minimizing Policy}\label{Sec: Safe control}
In this section, we convert the chance-constrained SOC problem to a risk-minimizing SOC problem and show that an optimal risk-minimizing control policy can be synthesized by solving an HJB PDE with an appropriate boundary condition.
\subsection{Risk-Minimizing SOC Problem}\label{Section: Risk-Minimizing SOC problem}
We approach the risk-constrained SOC problem (\ref{CC-SOC}) via Lagrangian relaxation. Define a risk-minimizing cost function:
\begin{equation}\label{Chat_old}
\begin{aligned}
     \!\!\widehat{C}\!\left(x_0,\!t_0,\! {u}(\cdot)\right)\!\coloneqq C\!\left(x_0,\! t_0, \!u(\cdot)\right)+\!\eta\,P\!\left(\!\bigvee_{t\in[t_0, T]} \!\!\!\pmb{x}(t)\!\notin\! \mathcal{X}_{s}\!\!\right)
\end{aligned}
\end{equation}
where $\eta\geq 0$ is the Lagrange multiplier corresponding to the risk constraint in (\ref{CC-SOC}). Establishing the exact correspondence between $\eta$ and the risk tolerance $\Delta$ (e.g. the Lagrange multiplier theorem \cite{bertsekas1997nonlinear}) is not in the scope of this paper and left as a topic for future work. In what follows, we will treat $\eta$ as a given constant, and focus on the control synthesis to minimize (\ref{Chat_old}). In Section \ref{Sec: Simulation}, it is demonstrated that $\eta$ can be tuned appropriately to achieve a desired level of safety.\par 

Let us express $P_{fail}$ (\ref{pfail}) in terms of the exit time $\pmb{t}_f$ as follows:
\begin{equation}\label{pfail2}
    P\left(\bigvee_{t\in[t_0, T]} \pmb{x}(t)\notin \mathcal{X}_{s}\right)=\mathbb{E}_{x_0, t_0}\left[\mathds{1} _{\pmb{x}(\pmb{t}_f)\in \partial\mathcal{X}_{s}}\right].
\end{equation}
Define $\phi:\overline{\mathcal{X}_s}\to\mathbb{R}$ as\footnote{In the sequel, the function $\phi(x)$ sets a boundary condition for a PDE. There, we often need technical assumptions on the regularity of $\phi(x)$ (e.g., continuity on $\overline{\mathcal{X}_s}$) to guarantee the existence of a solution. When such requirements are needed, we approximate  (\ref{phi(x)}) as $\phi(x) \approx \psi(x)B(x) + \eta\left(1-B(x)\right)$, where $B(x)$ is a smooth bump function on $\mathcal{X}_s$.}: 
\begin{equation}\label{phi(x)}
    \phi\left({x}\right)\coloneqq\psi\left({x}\right)\cdot \mathds{1} _{{x}\in \mathcal{X}_{s}}+\eta\cdot\mathds{1} _{{x}\in \partial\mathcal{X}_{s}}.
\end{equation}
Combining (\ref{C}), (\ref{Chat_old}), (\ref{pfail2}), and (\ref{phi(x)}), we obtain
\begin{equation}\label{Chat}
    \widehat{C}\!\left(x_0,\!t_0,\! {u}(\cdot)\!\right)\!=\!\mathbb{E}_{x_0, t_0}\!\!\left[\!\phi\!\left(\pmb{x}(\pmb{t}_f)\!\right)\!+\!\!\!\int_{t_0}^{\pmb{t}_f}\!\!\!\left(\!\frac{1}{2}\pmb{u}^T\!R\pmb{u}\!+\!V\!\!\right)\!dt\!\right]\!.
\end{equation}
Now, the risk-minimizing SOC problem is formulated as follows:
\begin{problem} [Risk-minimizing SOC problem]\label{Problem: Risk-minimizing SOC problem}
\begin{equation}\label{risk-minimizing SOC problem (tf)}
\begin{aligned}
\argmin_{{u}(\cdot)} \;\; & \mathbb{E}_{x_0, t_0}\!\!\left[\phi\left(\pmb{x}(\pmb{t}_f)\right)\!+\int_{t_0}^{\pmb{t}_f}\left(\frac{1}{2}\pmb{u}^TR\pmb{u}+V\right)dt\right]\\
\textrm{s.t.} \;\; & d\pmb{x}={f}dt+{g}\,{u}dt+{\sigma}d\pmb{w},\quad \pmb{x}(t_0)=x_0.
\end{aligned}    
\end{equation}
\end{problem}
Note that the cost function of Problem \ref{Problem: Risk-minimizing SOC problem} possesses the time-additive Bellman structure for the application of dynamic programming, without requiring any conservative approximations. 

\subsection{Risk-Minimizing Control Policy}\label{Section: Risk-Minimizing Control Policy}
In order to solve Problem \ref{Problem: Risk-minimizing SOC problem} via dynamic programming, for each $(x,t)\in\overline{\mathcal{Q}}$, and an admissible policy $u(\cdot)$ over $[t, T)$, we define the cost-to-go function:
\begin{equation}\label{value function}
    \widehat{C}\!\left(x, t, {u}(\cdot)\right)\!=\mathbb{E}_{x, t}\!\left[\phi\left(\pmb{x}(\pmb{t}_f)\right)\!+\!\!\int_{t}^{\pmb{t}_f}\!\!\!\left(\!\frac{1}{2}\pmb{u}^T\!R\pmb{u}+\!V\!\!\right)\!dt\right]\!\!.
\end{equation}
The following result is obtained by applying the verification theorem \cite[Chapter IV]{fleming2006controlled} to Problem \ref{Problem: Risk-minimizing SOC problem}.  

\begin{theorem}\label{solution to risk-minimizing soc}
Suppose there exists a continuous function $J:\overline{\mathcal{Q}}\rightarrow \mathbb{R}$ such that 
\begin{enumerate}[a)]
    \item $J(x,t)$ is continuously differentiable in $t$ and twice continuously differentiable in $x$ in the domain $\mathcal{Q}$;
    \item $J(x,t)$ solves the following dynamic programming PDE (HJB PDE):
    \begin{equation}\label{HJB PDE}
  \!\!\!\!\!\!\!\begin{cases}
     \begin{aligned}
         \!\!-\partial_tJ\!=&\!-\frac{1}{2}\!\left(\partial_xJ\right)^T\!\!gR^{-1}\!g^T\!\partial_xJ\!+\!V\\
         &+\!\!f^T\!\partial_xJ+\frac{1}{2}\text{Tr}\left(\sigma\sigma^T\partial^2_xJ\right),
          \end{aligned} &  (x,t)\in\mathcal{Q}, \\
    \!\!J(x,t)=\phi(x), & (x,t)\in\partial\mathcal{Q}.
  \end{cases}
    \end{equation}
\end{enumerate}
Then, the following statements hold:
\begin{enumerate}[1)]
\item $J(x,t)$ is the \textit{value function} for Problem \ref{Problem: Risk-minimizing SOC problem}. That is,
\begin{equation}\label{J as value function}
    J\left(x, t\right) = \min_{{u}(\cdot)} \widehat{C}\left(x, t, {u}(\cdot)\right),\quad \forall\;(x,t)\in\overline{\mathcal{Q}}.
\end{equation}
\item The solution to Problem \ref{Problem: Risk-minimizing SOC problem} is given by 
\begin{equation}\label{optimal policy}
    u^*(x,t)=-R^{-1}\left(x, t\right){g}^T\left(x, t\right)\partial_xJ\left(x, t\right).
\end{equation}
\end{enumerate}
\end{theorem}
\begin{proof}
Let $J(x,t)$ be the function satisfying a) and b). By the fundamental theorem of calculus and the It\^{o} formula, for each $(x,t)$ we have
\begin{equation}\label{ito}
    \begin{aligned}
       \mathbb{E}_{x,t}\left[J\left(\pmb{x}(\pmb{t}_f), \pmb{t}_f\right)\right]&=J(x,t)\\ &\!\!\!\!\!\!\!\!\!\!\!\!\!\!\!\!\!\!\!\!\!\!\!\!\!\!\!\!\!\!\!\!\!\!\!\!\!\!\!\!\!\!\!\!\!\!\!\!\!\!\!+\mathbb{E}_{x,t}\!\!\left[\int_{t}^{\pmb{t}_f}\!\!\!\left(\!\!\partial_tJ \!+\! (f\!+\!gu)^T\!(\partial_xJ)\!+\!\frac{1}{2}\text{Tr}\!\left(\!\sigma\sigma^T\!\partial_x^2J\right)\!\!\right)\!ds\right]\!.
    \end{aligned}
\end{equation}
By the boundary condition of the PDE (\ref{HJB PDE}), $\mathbb{E}_{x,t}\left[J\left(\pmb{x}(\pmb{t}_f), \pmb{t}_f\right)\right] = \mathbb{E}_{x,t}\left[\phi\left(\pmb{x}(\pmb{t}_f)\right)\right]$. Hence, from (\ref{ito}), we obtain
\begin{equation}\label{ito after plugging BC}
    \begin{aligned}
       J(x,t)&=\mathbb{E}_{x,t}\left[\phi\left(\pmb{x}(\pmb{t}_f)\right)\right]\\ &\!\!\!\!\!\!\!\!\!\!\!\!\!\!\!\!\!\!\!\!\!\!-\mathbb{E}_{x,t}\!\!\left[\int_{t}^{\pmb{t}_f}\!\!\!\left(\!\!\partial_tJ \!+\! (f\!+\!gu)^T\!(\partial_xJ)\!+\!\frac{1}{2}\text{Tr}\!\left(\!\sigma\sigma^T\!\partial_x^2J\right)\!\!\right)\!ds\right]\!.
    \end{aligned}
\end{equation}
Now, notice that the right hand side of the PDE from (\ref{HJB PDE}) can be expressed as the minimum value of a quadratic form in $u$ as follows:
\begin{equation*}
    \begin{aligned}
         \!-\partial_tJ\!=\!\min_{{u}}\!\left[\frac{1}{2}u^T\!Ru\!+\!V\!+\!\left(\!f\!+\!gu\right)^T\!\partial_xJ\!
         +\!\frac{1}{2}\text{Tr}\!\left(\!\sigma\sigma^T\partial^2_xJ\right)\!\right]\!.
          \end{aligned} 
\end{equation*}
Therefore, for an arbitrary $u$, we have 
\begin{equation}\label{pde ineqaulity}
    \begin{aligned}
         \!-\partial_tJ\!\leq\frac{1}{2}u^T\!Ru\!+\!V\!+\!\left(f\!+\!gu\right)^T\!\partial_xJ\!
         +\!\frac{1}{2}\text{Tr}\left(\sigma\sigma^T\partial^2_xJ\right)
          \end{aligned} 
\end{equation}
where the equality holds iff
\begin{equation}\label{optimal policy2}
    u = -R^{-1}g^T\partial_xJ.
\end{equation}
Combining (\ref{ito after plugging BC}) and (\ref{pde ineqaulity}), we obtain
\begin{equation}\label{J leq C_hat}
  \begin{aligned}
       \!J(x,t)&\leq\mathbb{E}_{x,t}\left[\phi\left(\pmb{x}(\pmb{t}_f)\right)\right]+\mathbb{E}_{x,t}\!\left[\!\int_{t}^{\pmb{t}_f}\!\!\!\left(\!\frac{1}{2}\pmb{u}^T\!R\pmb{u}\!+\!V\!\!\right)\!\!ds\!\right]\\
       &= \widehat{C}\left(x,t,{u}(\cdot)\right).
    \end{aligned}   
\end{equation}
This proves the statement 1). Since (\ref{J leq C_hat}) holds with equality iff (\ref{optimal policy2}) is satisfied, the statement 2) also follows. 
\end{proof}
General conditions under which there exists a function $J(x,t)$ satisfying a) and b) are beyond the scope of this paper. However, in Section \ref{Section: Numerical Methods} below, we focus on a special case in which (\ref{HJB PDE}) can be linearized, where the existence of such a function is guaranteed. Since the boundary condition of the HJB PDE (\ref{HJB PDE2}) is characterized by the Lagrange multiplier of Problem \ref{Problem: Risk-minimizing SOC problem}, Theorem \ref{solution to risk-minimizing soc} implies that the safety of an optimal policy of Problem \ref{Problem: Risk-minimizing SOC problem} can be tuned through the boundary condition of the HJB PDE.  

\section{Risk Estimation}\label{Sec: Risk Analysis}
In this section, we show that the proposed risk-minimizing control problem can be viewed as a generalization of the risk estimation problem. The risk estimation problem can be stated as follows:
\begin{problem}[Risk estimation]\label{Problem: Risk estimation}
Find the probability of failure ($P_{fail}$) defined as (\ref{pfail2}), for the system (\ref{SDE}), given an admissible control policy ${u}\left(\cdot\right)$.
\end{problem}
Let us plug the following values 
\begin{equation}\label{replacement}
    \psi(x)\equiv0, \;\; R(x,t)\equiv0, \;\; V(x,t)\equiv0, \;\; \eta=1
\end{equation}
in (\ref{phi(x)}) and (\ref{Chat}), and define new cost functions\footnote{As mentioned in Section \ref{Section: Risk-Minimizing SOC problem}, similar to $\phi(x)$, we need assumptions on the regularity of $\widetilde{\phi}(x)$ in the sequel. These can be met by approximating $\widetilde{\phi}(x)$ with the help of a smooth bump function.}:
\begin{equation*}
\!\!\widetilde{\phi}(x) \!=\! \mathds{1} _{{x}\in \partial\mathcal{X}_{s}}, \quad \widetilde{C}\left(x_0, t_0, u(\cdot)\right) \!=\! \mathbb{E}_{x_0, t_0}\!\!\left[\mathds{1} _{\pmb{x}(\pmb{t}_f)\in \partial\mathcal{X}_{s}}\right].
\end{equation*}
From (\ref{pfail2}), we recognize that 
\begin{equation}\label{Pfail as Ctilde}
    \widetilde{C}\left(x_0, t_0, u(\cdot)\right)=P_{fail}.
\end{equation}
Hence, computing $\widetilde{C}$ for a given admissible control policy $u(\cdot)$ is in fact the risk estimation problem. For this purpose, we define the cost-to-go function for each $(x,t)\in\overline{\mathcal{Q}}$:
\begin{equation}\label{value function for risk}
\widetilde{C}\left(x, t, u(\cdot)\right) = \mathbb{E}_{x, t}\left[\mathds{1} _{\pmb{x}(\pmb{t}_f)\in \partial\mathcal{X}_{s}}\right].
\end{equation}
Now, we show that Problem \ref{Problem: Risk estimation} can be solved via a PDE which is a special case of the HJB PDE (\ref{HJB PDE}). Corollary \ref{Theorem: risk estimation} follows from Theorem \ref{solution to risk-minimizing soc}. 
\begin{corollary}\label{Theorem: risk estimation}
Suppose there exists a continuous function $J:\overline{\mathcal{Q}}\rightarrow \mathbb{R}$ that satisfies a), and solves the following PDE:
      \begin{equation}\label{risk PDE}
      \begin{cases}
 \begin{aligned}
    \!\!-\partial_{t}J\!=\!\left(\!f\!+\!gu\right)^T\!\!\partial_xJ\!+\!\frac{1}{2}\text{Tr}\!\left(\!\sigma\sigma^T\!\partial^2_xJ\right)\!,
    \end{aligned}& \!\!(x,t)\!\in\!\mathcal{Q}, \\
    \!\!J(x,t)= \widetilde{\phi}(x), & \!\!(x,t)\!\in\!\partial\mathcal{Q}.
    \end{cases}
          \end{equation}
Then, the solution to Problem \ref{Problem: Risk estimation} is given by
\begin{equation*}
    P_{fail} = J(x_0,t_0).
\end{equation*}
\end{corollary}
\begin{proof}
Let $J(x,t)$ be a function that satisfies a) and solves the HJB PDE: 
\begin{equation}\label{HJB PDE2}
\begin{cases}
\begin{aligned}
         -\partial_tJ\!=\!&\min_{{u}}\!\!\Bigg[\!\frac{1}{2}u^T\!Ru\!+\!\left(f\!\!+\!gu\right)^T\!\!\partial_xJ\\
         &+V+\frac{1}{2}\text{Tr}\left(\sigma\sigma^T\partial^2_xJ\right)\Bigg],
          \end{aligned} &  (x,t)\in\mathcal{Q},\\
    J(x,t)=\phi(x), & (x,t)\in\partial\mathcal{Q}. 
  \end{cases}
\end{equation}
Then from Theorem \ref{solution to risk-minimizing soc}, $J(x,t)$ satisfies (\ref{J as value function}) for the cost function $\widehat{C}$ defined in (\ref{Chat}). Now, using the values from (\ref{replacement}), the HJB PDE (\ref{HJB PDE2}) becomes (\ref{risk PDE}), and 
\begin{equation}\label{J(x,t)}
    J(x,t)=\widetilde{C}\left(x, t, u(\cdot)\right).
\end{equation}
Note that, here we do not perform minimization over $u$, since the control policy is already known to us. Hence, if $J(x,t)$ satisfies (\ref{risk PDE}), from (\ref{J(x,t)}), and (\ref{Pfail as Ctilde}), we obtain $P_{fail}=J(x_0, t_0)$.   
\end{proof}
 Corollary \ref{Theorem: risk estimation} implies that Problem \ref{Problem: Risk-minimizing SOC problem} can be viewed as a generalization of Problem \ref{Problem: Risk estimation}, and PDE (\ref{risk PDE}) is a special case of the HJB PDE (\ref{HJB PDE}). Note that (\ref{HJB PDE}) is nonlinear in $J(x,t)$, whereas (\ref{risk PDE}) is linear. PDE (\ref{risk PDE}) is also a special case of the Dirichlet-Poisson problem \cite[Chapter 9]{oksendal2013stochastic}, where the existence and the uniqueness of a solution can be guaranteed under a sufficiently regular boundary condition. Corollary \ref{Theorem: risk estimation} is consistent with the results obtained in \cite[Theorem 1]{chern2021safe} for the quantification of safety.  

\section{Numerical Methods}\label{Section: Numerical Methods}
It is necessary to resort to numerical methods to solve the HJB PDE (\ref{HJB PDE}) approximately since an analytical solution is in general not available. In this section, we restrict our attention to a class of SOC problems for which the associated HJB equation can be linearized and consider two numerical methods, namely, the path integral control and FDM, to solve the linearized PDE. Consider the class of SOC problem for which there exists a positive constant $\lambda$ satisfying the following equation: 
\begin{equation}\label{lambda}
  \sigma(x, t)\sigma^T(x, t) = \lambda g(x, t)R^{-1}(x,t) g^T(x, t). 
\end{equation}
This condition implies that the control input in the directions with a higher noise variance is cheaper than that in the direction with a lower noise variance. See \cite{kappen2005path} for further discussion on this condition. Using the constant $\lambda$ satisfying (\ref{lambda}), introduce the Cole-Hopf transformation
\begin{equation}\label{exp transformation}
 J(x,t) = -\lambda\,\text{log}\left(\xi\left(x,t\right)\right).
\end{equation}
The transformation (\ref{exp transformation}) along with the condition (\ref{lambda}) allows us to rewrite the PDE (\ref{HJB PDE}) as a linear PDE in terms of $\xi\left(x,t\right)$:
\begin{equation}\label{linearized risk-minimizing HJB}
 \begin{cases}
     \!\partial_t\xi\!=\!\frac{V\xi}{\lambda}\!-\!f^T\partial_x\xi-\frac{1}{2}\text{Tr}\left(\sigma\sigma^T\partial^2_x\xi\right),       & (x,t)\in\mathcal{Q}, \\
    \!\xi(x,t)\!=\!\text{exp}\left(-\frac{\phi(x)}{\lambda}\right), & (x,t)\in\partial\mathcal{Q}.\\  
  \end{cases}
\end{equation}
Similar to (\ref{risk PDE}), the PDE (\ref{linearized risk-minimizing HJB}) is also a special case of the Dirichlet-Poisson problem, where the existence and the uniqueness of a solution can be guaranteed under a sufficiently regular boundary condition.
\subsection{Path Integral}
\begin{figure*}[t]
    \centering
      \begin{tabular}{c c c c}
      \!\!\!\!\!\!\!\!\!\!\includegraphics[scale=0.3]{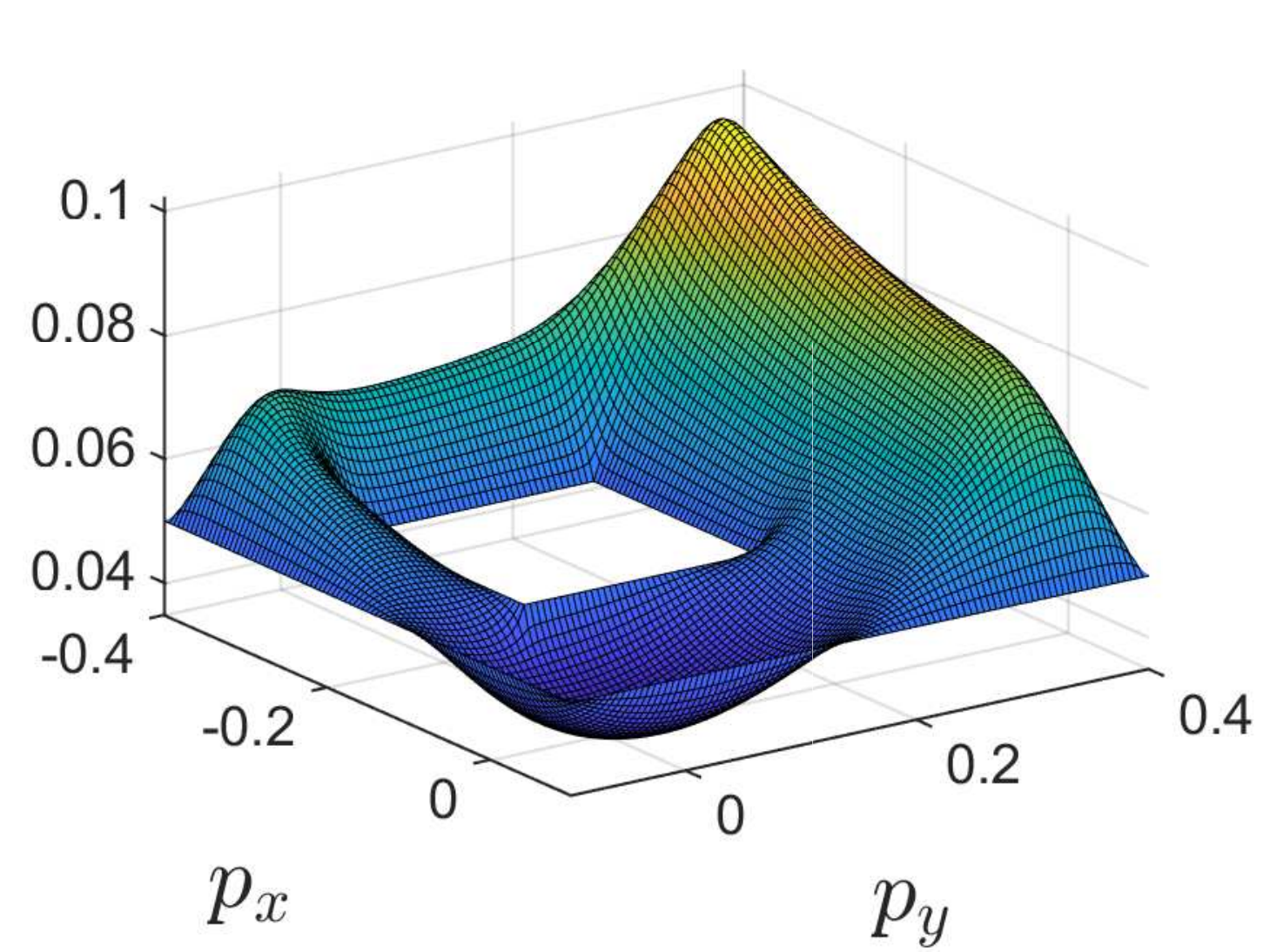} &\!\!\!\!\!\!\!\!\includegraphics[scale=0.3]{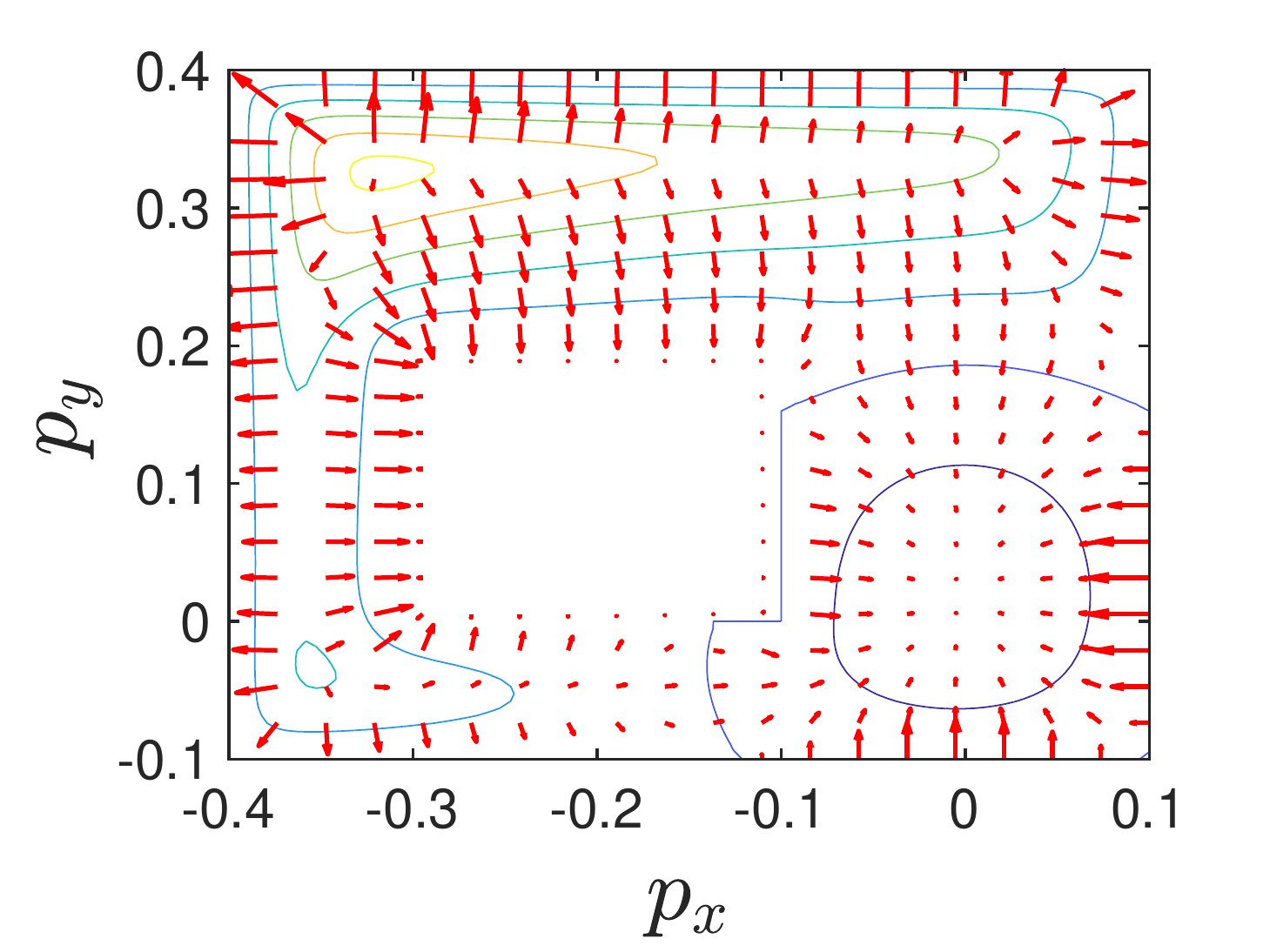} &\!\!\!\!\!\!\!\!\includegraphics[scale=0.3]{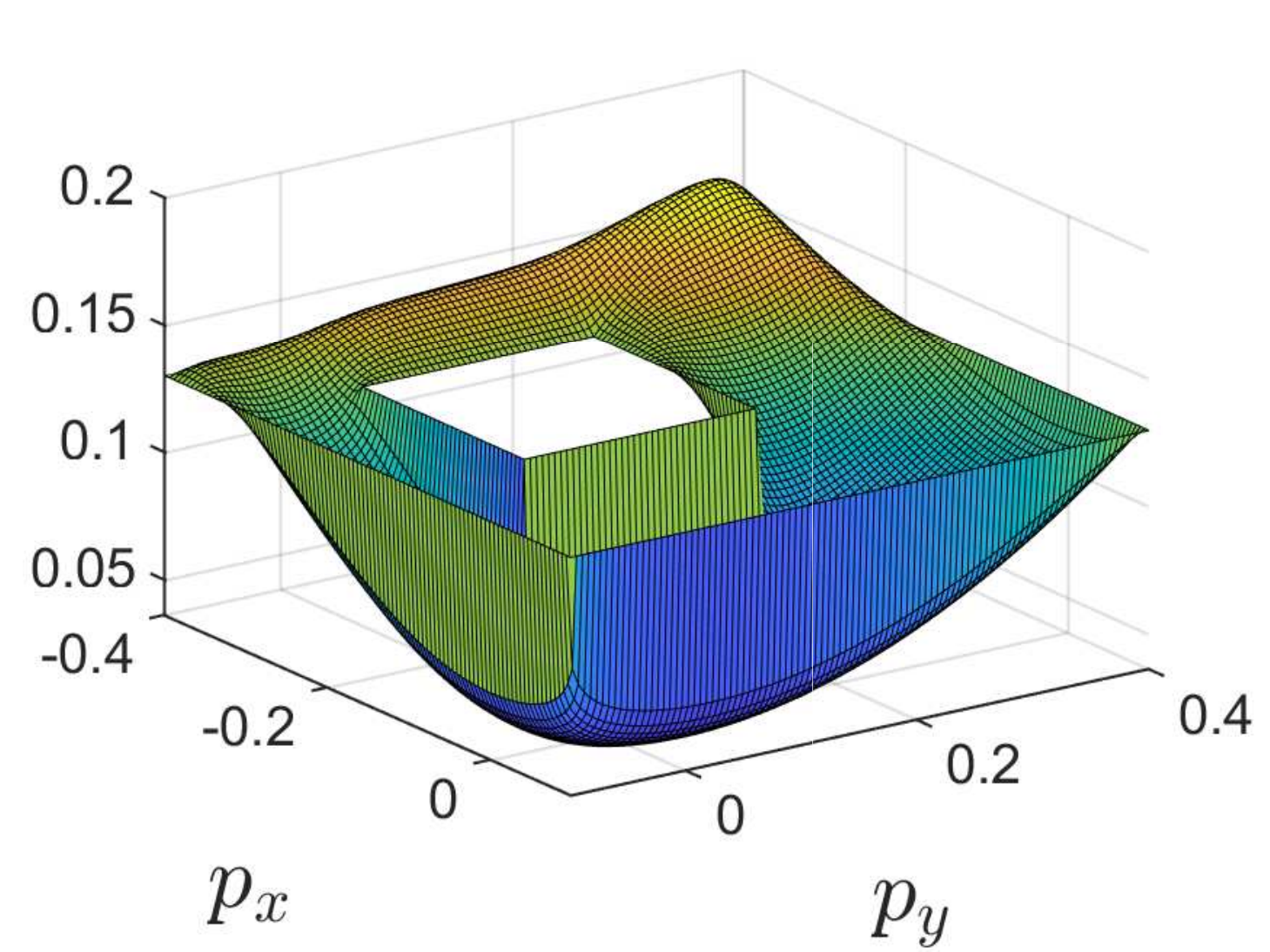} &\!\!\!\!\!\!\!\!\!\! \includegraphics[scale=0.3]{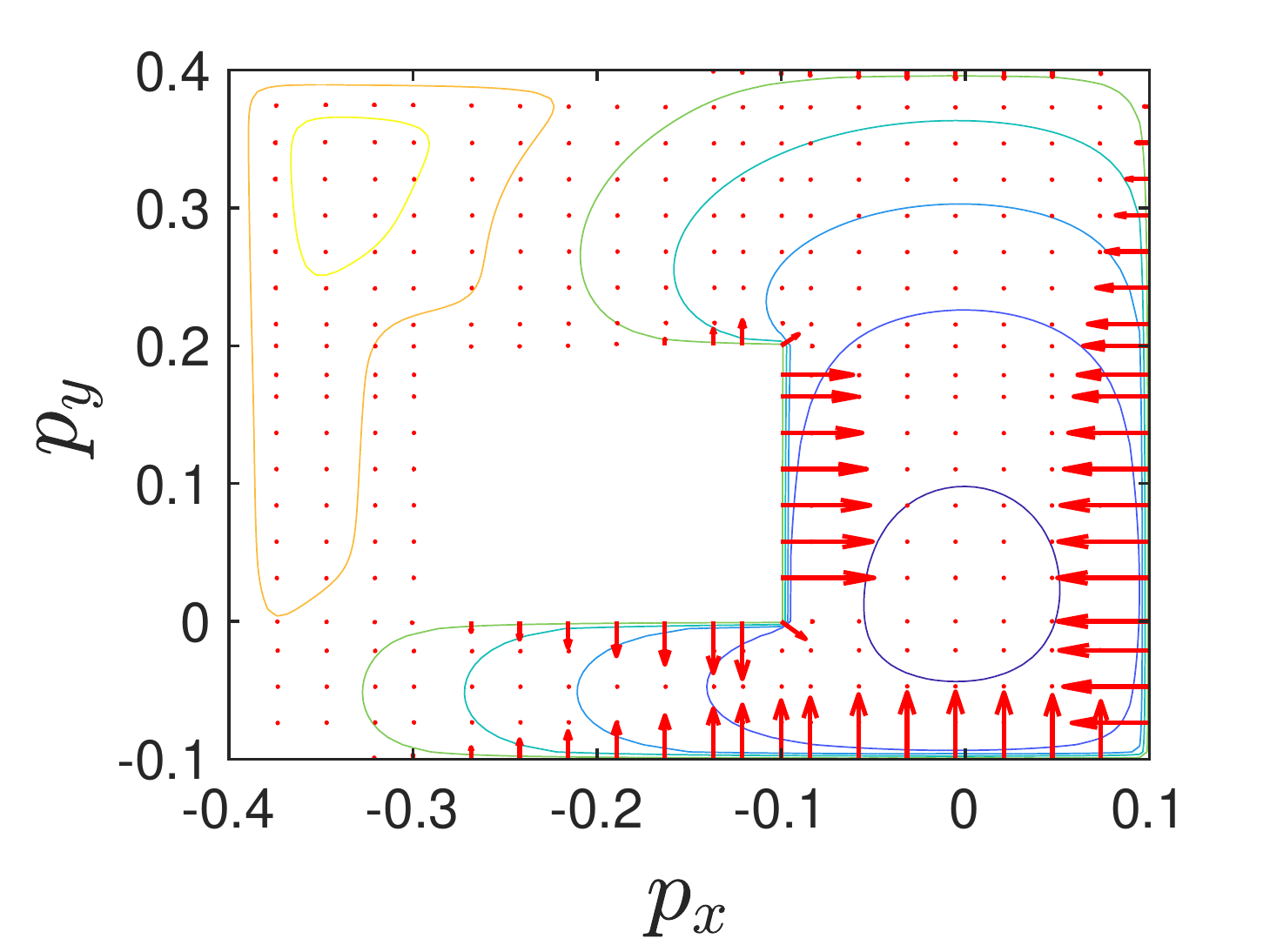}\\
      (a) $J(x,t_0)$ for $\eta = 0.05$ & (b) $u^*(x,t_0)$ for $\eta = 0.05$ & (c) $J(x,t_0)$ for $\eta = 0.13$ & (d) $u^*(x,t_0)$ for $\eta = 0.13$\\
      
        \!\!\!\!\!\!\!\!\!\!\includegraphics[scale=0.3]{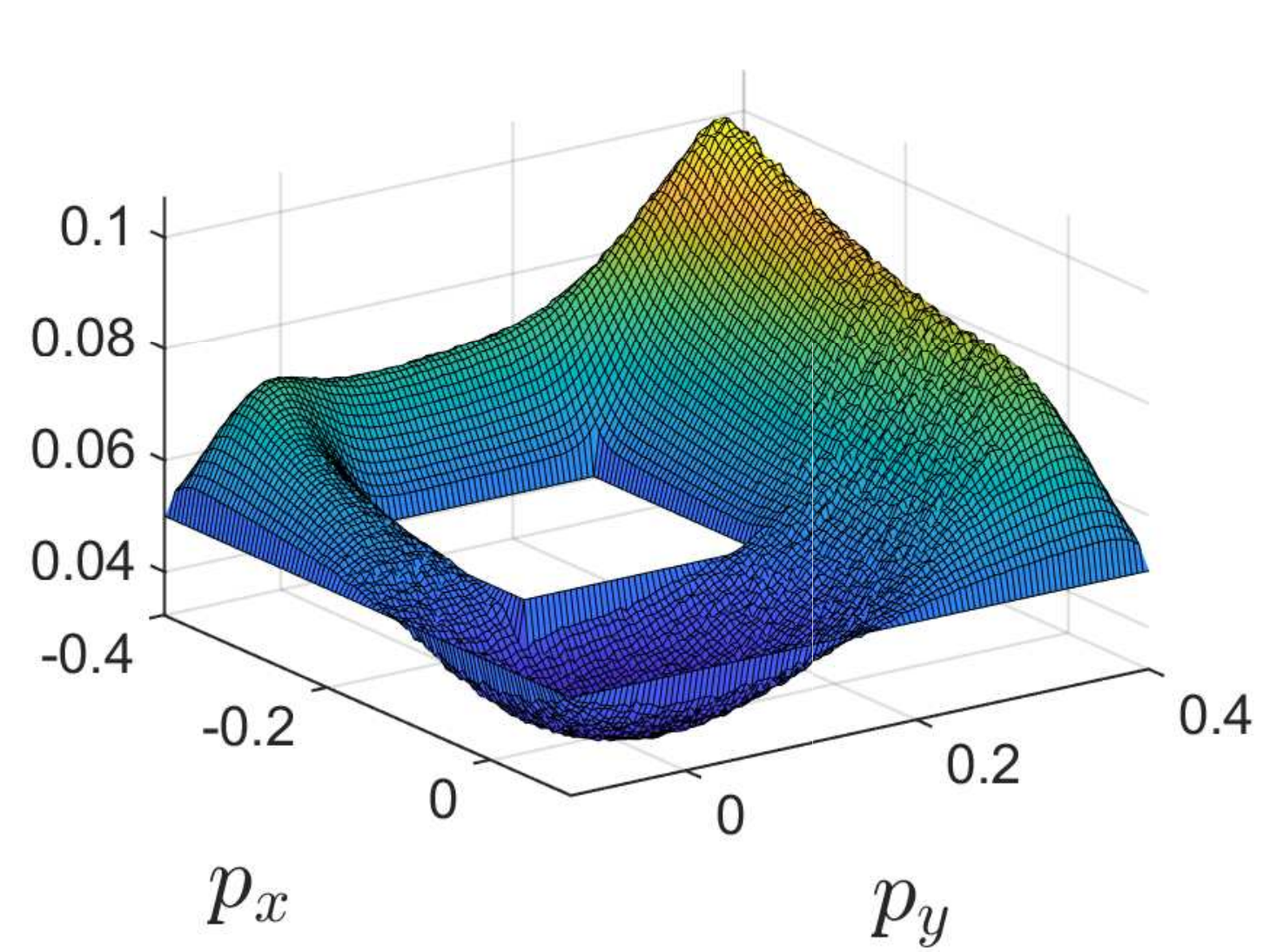} &\!\!\!\!\!\!\!\!\includegraphics[scale=0.3]{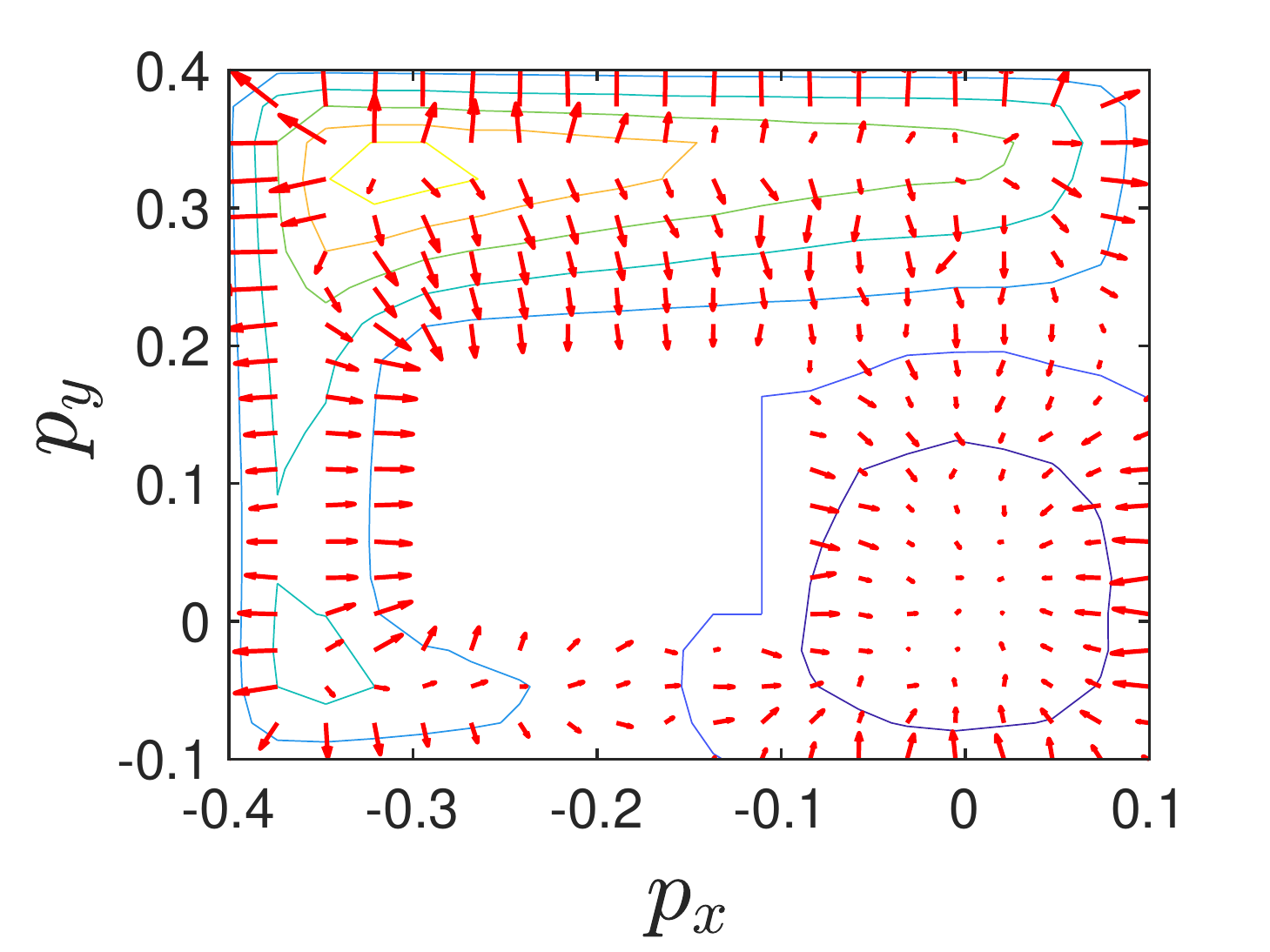}&\!\!\!\!\!\!\!\!\includegraphics[scale=0.3]{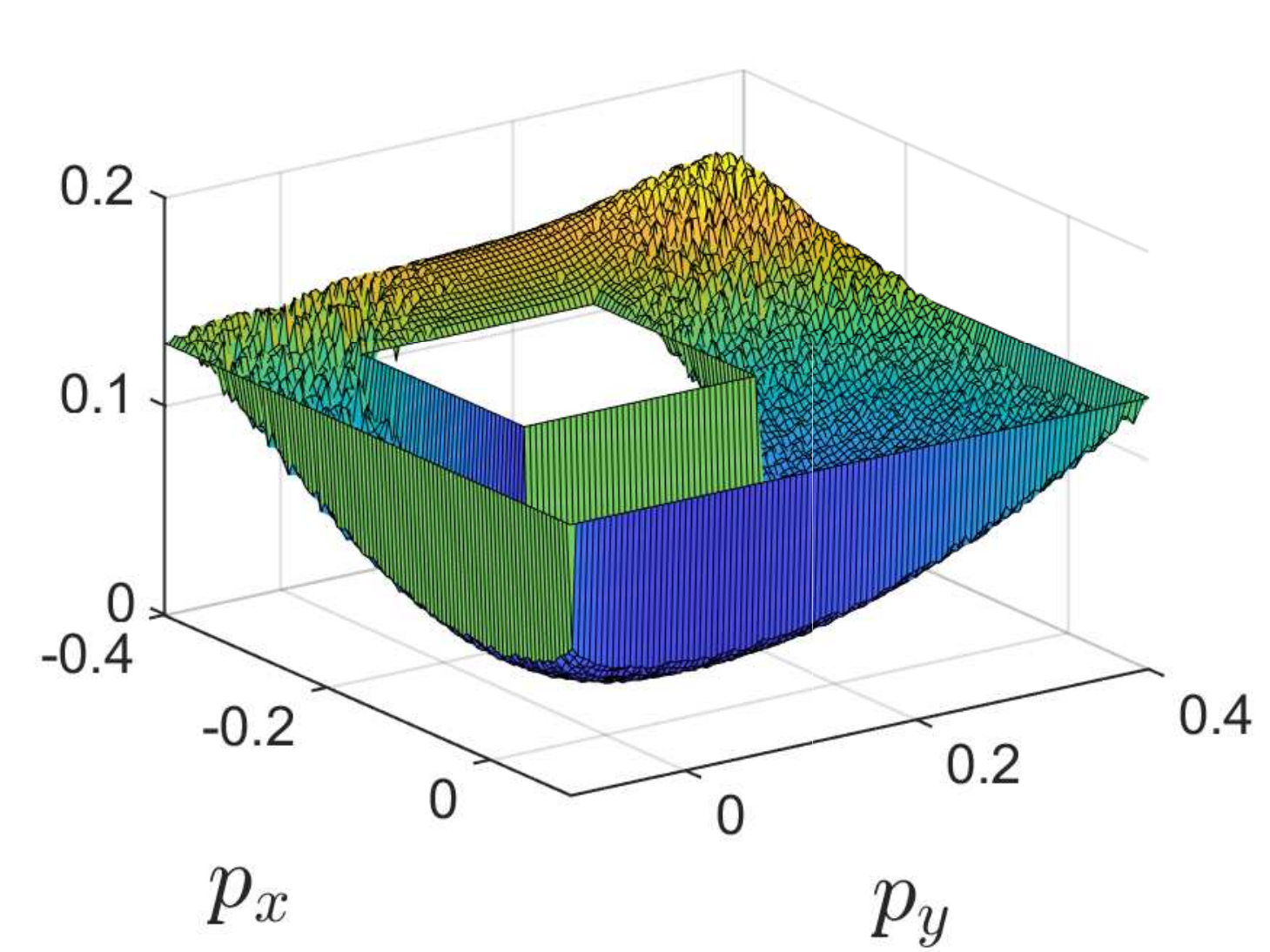} &\!\!\!\!\!\!\!\!\!\!\includegraphics[scale=0.3]{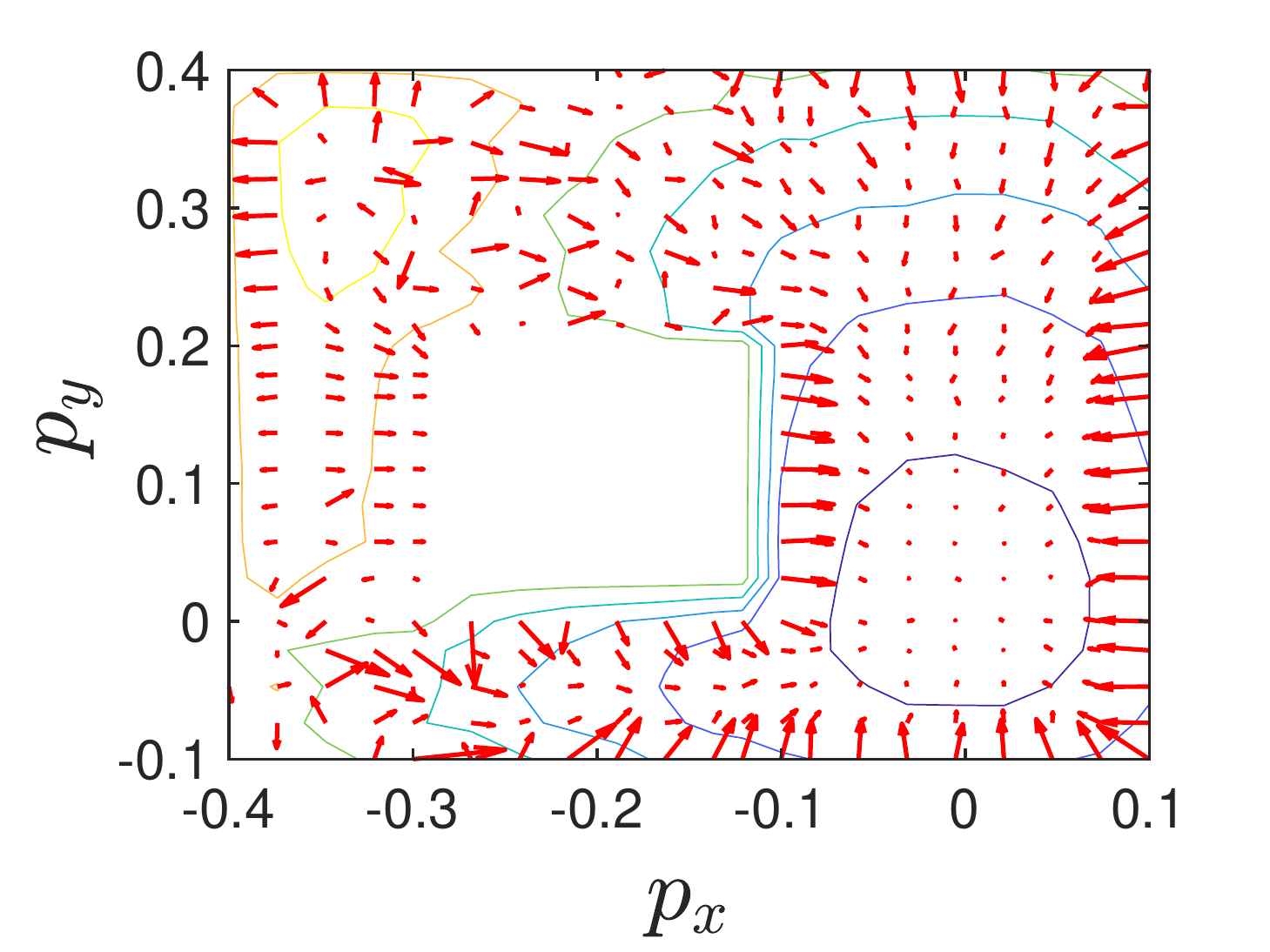}\\
        (e) $J(x,t_0)$ for $\eta = 0.05$ & (f) $u^*(x,t_0)$ for $\eta = 0.05$ & (g) $J(x,t_0)$ for $\eta = 0.13$ & (h) $u^*(x,t_0)$ for $\eta = 0.13$\\
       \end{tabular}
        \caption{Comparison of the solutions of $J(x,t_0)$ and $u^*(x,t_0)$ obtained from the FDM (a-d) and the path integral (e-h) for $\eta=0.05$ and $\eta=0.13$. The optimal control inputs $u^*(x,t_0)$ in (b, d, f, h) are plotted with the contours of $J(x, t_0)$.} 
        \label{Fig. surf plots FDM}
\end{figure*}
In the path integral control framework, we utilize the fact that the solution to the PDE (\ref{linearized risk-minimizing HJB}) admits the Feynman-Kac representation \cite{oksendal2013stochastic}, \cite{yong1997relations}:  
\begin{equation}\label{xi}
    \xi\left(x,t\right)=\mathbb{E}\left[\text{exp}\left(-\frac{1}{\lambda}S\left(\tau\right)\right)\right]
\end{equation}
where $S\left(\tau\right)=\phi\left(\pmb{x}(\pmb{t}_f)\right)+\int_{t}^{\pmb{t}_f} V\left(\pmb{x}(t), t\right)dt$ is the cost-to-go of an uncontrolled trajectory $\tau$ of the system (\ref{SDE}) starting at $(x,t)$. The optimal control input $u^*(x,t)$ can be obtained by taking the gradient of $\xi(x,t)$ with respect to $x$ \cite{williams2017model}, \cite{theodorou2010generalized}. First, let us partition the state vector into $\pmb{x}(t)=\begin{bmatrix}\pmb{x}_p(t) &\pmb{x}_c(t)\end{bmatrix}$ with $\pmb{x}_p(t)\in\mathbb{R}^l$ the non-directly controlled part, and $\pmb{x}_c(t)\in\mathbb{R}^{n-l}$ the directly controlled part. Subsequently, $g(\pmb{x}(t),t)$ and $\sigma(\pmb{x}(t),t)$ can be partitioned as $g(\pmb{x}(t),t)=\begin{bmatrix}0_{l\times m}&g_c(\pmb{x}(t),t) \end{bmatrix}^T$ and $\sigma(\pmb{x}(t),t)=\begin{bmatrix}{0_{l\times k}}&\sigma_c(\pmb{x}(t),t) \end{bmatrix}^T$, where $g_c(\pmb{x}(t),t)$ and $\sigma_c(\pmb{x}(t),t)$ correspond to the directly actuated states. After taking the gradient of $\xi(x,t)$ with respect to $x$, we obtain
\begin{equation}\label{path integral control}
 u^*(x,t)dt=\mathcal{G}\!\left(x,t\right)\!\frac{\mathbb{E}\!\left[\text{exp}\!\left(\!-\frac{1}{\lambda}S\left(\tau\right)\right)\sigma_c\left(x,t\right)d\pmb{w}(t)\right]}{\mathbb{E}\left[\text{exp}\left(-\frac{1}{\lambda}S\left(\tau\right)\right)\right]}   
\end{equation}
where the matrix $\mathcal{G}\left(x,t\right)$ is defined as
\begin{equation*}
\mathcal{G}\left(x,t\right)=R^{-1}(x,t)g_c^T(x,t)\left(g_c(x,t)R^{-1}(x,t)g_c^T(x,t)\right)^{-1}. 
\end{equation*}\par
For computation, we can discretize the system dynamics (\ref{SDE}) and use Monte Carlo sampling to evaluate the expectations in (\ref{xi}) and (\ref{path integral control}) numerically \cite{williams2017model}. The path integral framework evaluates a solution of the HJB equation locally without requiring the solution nearby so that there is no need for a (global) grid of the domain. However, for real-time control, the Monte Carlo simulation must be performed in real-time to evaluate (\ref{path integral control}) for the current $(x,t)$. 
\subsection{Finite Difference Method}\label{Sec: FDM}
The finite difference method is routinely used to solve PDEs. When the geometry of the computational domain is simple, the implementation of the method is straightforward and high orders of spatial accuracy can be obtained with relatively low efforts \cite{Grossmann2007}. The numerical solution to (\ref{linearized risk-minimizing HJB}) is sought on a tensor grid defined in $\overline{\mathcal{X}_{s}}$ and finite difference formulas for first, second, and mixed spatial derivatives are derived from Taylor series. In this work, the discrete differential operators are formed using centered finite differences with up to eighth-order accuracy. For points near the boundary $\partial\mathcal{X}_{s}$, asymmetric finite difference formulas of matching order are used together with the Dirichlet boundary condition. Discretization of the spatial differential operators yields a system of linear ordinary differential equations (ODEs). This system of ODEs is solved using a commercially available implicit solver ode15s \cite{Shampine_1997_matlab}, which provides an efficient framework with high temporal order of accuracy, and variable time-stepping with error control.\par 
The FDM solution of the PDE (\ref{linearized risk-minimizing HJB}) can be computed off-line and an optimal control policy can be stored in a look-up table. For real-time control, the optimal input for a given state can be immediately obtained on-line through this look-up table. However, the complexity and memory requirements of the off-line computation increase exponentially as the state-space dimension increases.

\section{Simulation Results}\label{Sec: Simulation}
In this section, we validate the proposed control synthesis framework using a 2D robot navigation example. We show how the value of the Lagrange multiplier $\eta$ affects the safety of the synthesized optimal control policy and also compare the results obtained by FDM and path integral control. The problem is illustrated in Fig. \ref{Fig. sample trajs} where a particle robot wants to travel in a 2D space from a given start position (shown by the green star) to the origin. The white region trapped between the outer and inner rectangles is $\mathcal{X}_{s}$, and the edges of the rectangles (shown in red) represent $\partial\mathcal{X}_{s}$: the boundary of $\mathcal{X}_{s}$. The region inside the inner boundary (shown in black hatching), and outside the outer boundary (not shown in the figure) represents $\overline{\mathcal{X}_{s}}^c=\mathbb{R}^n\backslash\overline{\mathcal{X}_{s}}$. The states of the system are $\pmb{p}_x$ and $\pmb{p}_y$, the positions along $x$ and $y$ respectively. The system dynamics are given by the following SDEs:
\begin{equation}\label{velocity input model}
\begin{split}
    d\pmb{p}_x={v}_xdt+\sigma d\pmb{w}_x, \quad d\pmb{p}_y={v}_ydt+\sigma d\pmb{w}_y,
\end{split}
\end{equation}
where ${v}_x$ and ${v}_y$ are velocities along $x$ and $y$ directions respectively. Assume 
\begin{equation*}
 v_x = \overline{v}_x + \widetilde{v}_x, \qquad  v_y = \overline{v}_y + \widetilde{v}_y, 
\end{equation*}
where $\overline{v}_x$ and $\overline{v}_y$ are nominal velocities given by $\overline{v}_x = -k_xp_x $ and $\overline{v}_y = -k_yp_y$ for some constants $k_x$ and $k_y$. Hence, (\ref{velocity input model}) can be rewritten as
\begin{equation*}
\begin{split}
    d\pmb{p}_x&=-k_x\pmb{p}_xdt+\widetilde{{v}}_xdt+\sigma d\pmb{w}_x,\\ d\pmb{p}_y&=-k_y\pmb{p}_ydt+\widetilde{{v}}_ydt+\sigma d\pmb{w}_y.
\end{split}
\end{equation*}
Now, the goal is to design an optimal control policy $u^* = \begin{bmatrix}\widetilde{v}^*_x&\widetilde{v}^*_y\end{bmatrix}^T$ for the cost function given in (\ref{risk-minimizing SOC problem (tf)}). We set $k_x=k_y=0.5$, $V\left(\pmb{x}(t)\right) = \pmb{p}_x^2(t) + \pmb{p}_y^2(t)$, $\psi\left(\pmb{x}(T)\right) = \pmb{p}_x^2(T) + \pmb{p}_y^2(T)$, $R=1$, $t_0=0$, and $T=2$. We use $\sigma^2=0.01$ except in Fig. \ref{Fig. pfail with s2}. For the FDM simulation, the domain is discretized into a $96\times 96$ grid and the implicit solver ode15s is used with a relative error tolerance of $0.001$. In the path integral simulation, for Monte Carlo sampling, $10000$ trajectories and a step size of $0.01$ are used. \par   
\begin{figure}[t]
    \centering
      \begin{tabular}{c c}
     \!\!\!\!\!\!\!\!\!\includegraphics[scale=0.35]{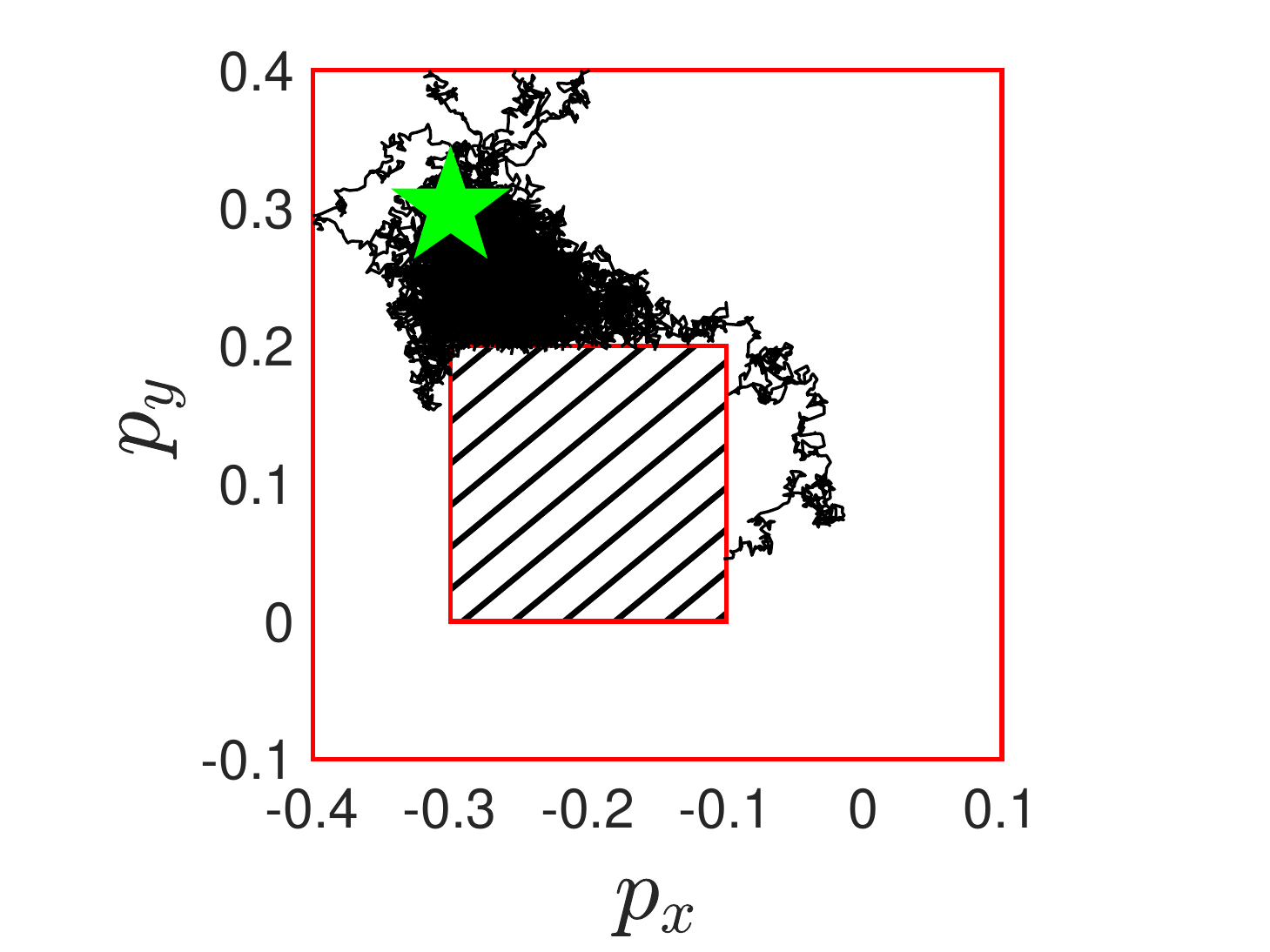} &\!\!\!\!\!\!\!\!\!\!\!\!\!\!\!\!\!\!\!\!\!\!\!\!\includegraphics[scale=0.35]{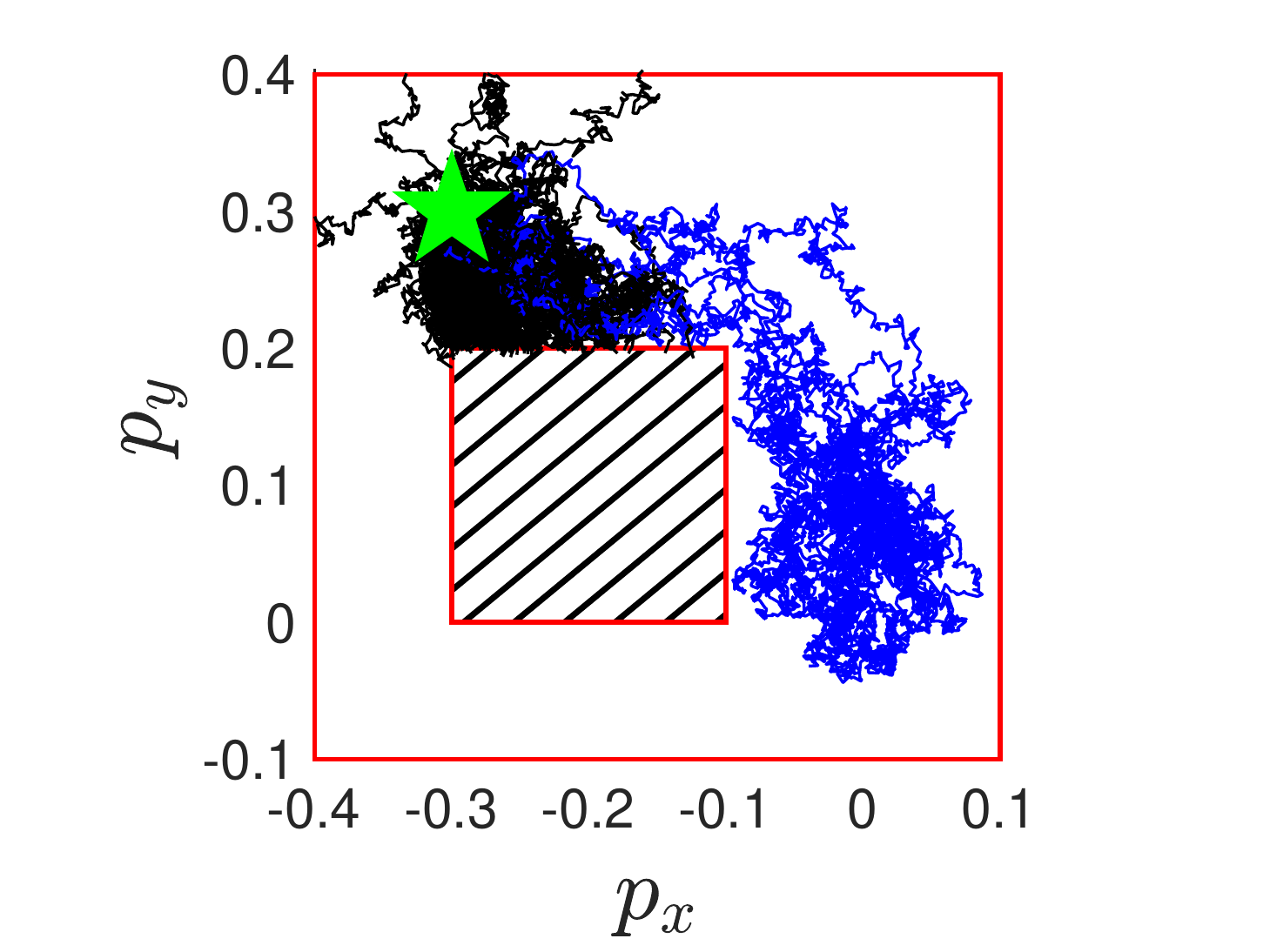} \\
      (a) $\eta = 0.05$ & \!\!\!\!\!\!\!\!(b) $\eta = 0.09$\\
       \!\!\!\!\!\!\!\!\!\includegraphics[scale=0.35]{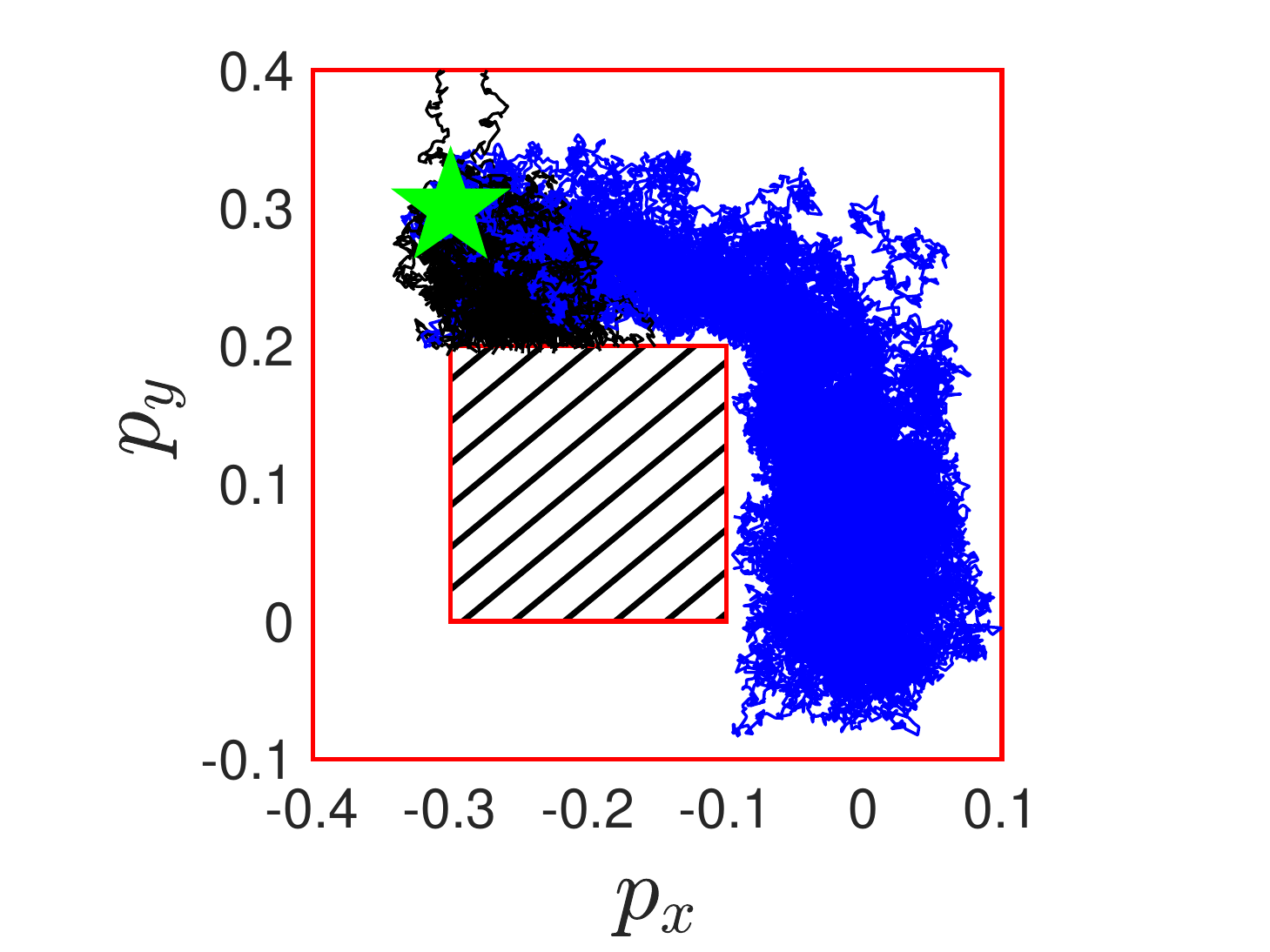} &\!\!\!\!\!\!\!\!\!\!\!\!\!\!\!\!\!\!\!\!\!\!\!\! \includegraphics[scale=0.35]{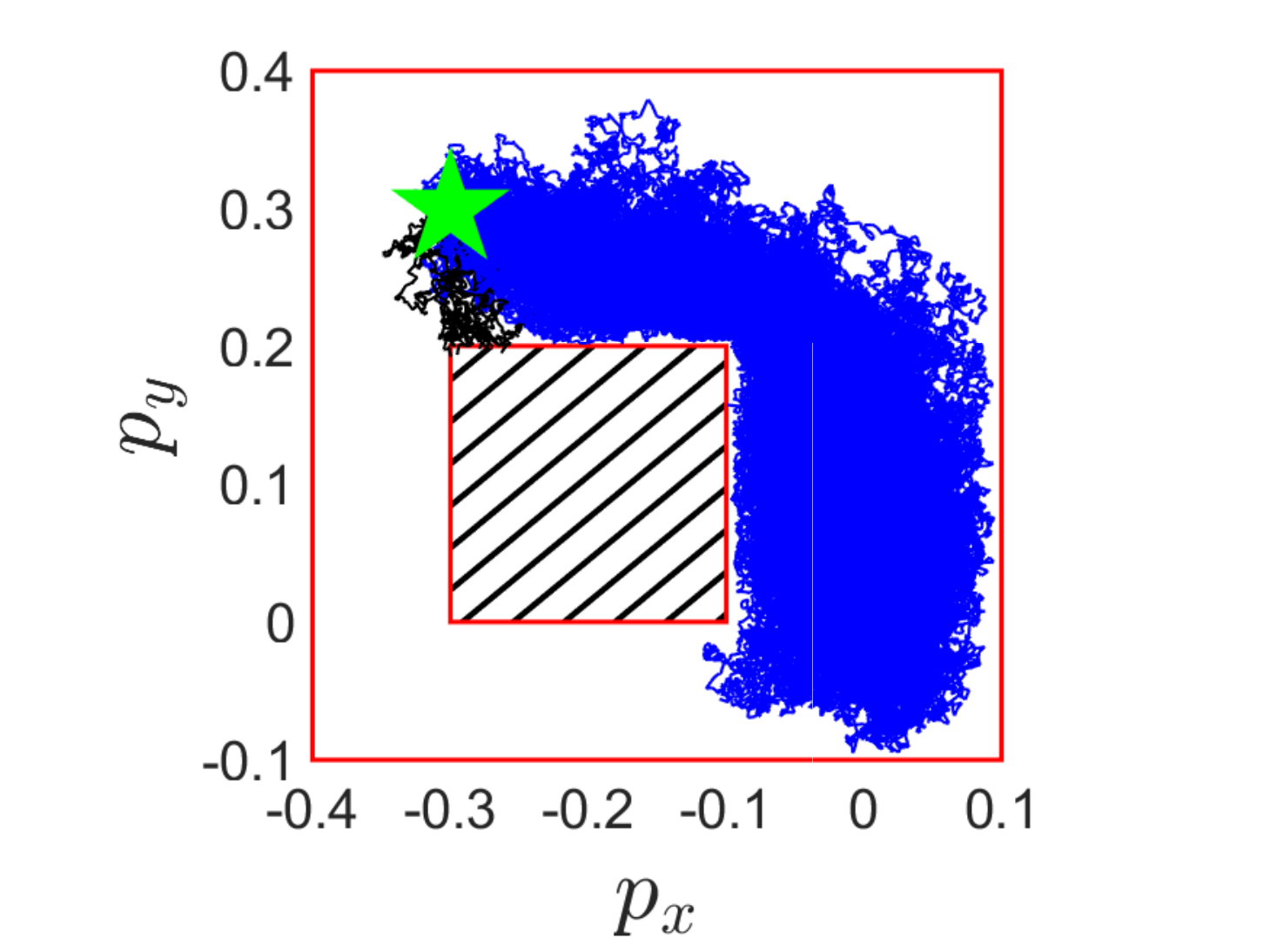}\\
        (c) $\eta = 0.11$ & \!\!\!\!\!\!\!\!(d) $\eta = 0.13$\\
       \end{tabular}
        \caption{A 2D robot navigation problem. The start position is shown by a green star and the target position is at the origin. $100$ sample trajectories generated using optimal control policies for four different values of $\eta$ are shown. The trajectories are color-coded; black paths collide with the inner or outer boundary, while blue paths converge in the neighborhood of the origin.} 
        \label{Fig. sample trajs}
\end{figure}
Fig. \ref{Fig. surf plots FDM} shows the comparison of the  solutions of $J(x,t_0)$ and $u^*(x,t_0)$ obtained from the FDM (top row) and the path integral (bottom row), for $\eta=0.05$ and $\eta=0.13$. Since the path integral is a sampling-based approach its solutions are noisier than FDM, as expected. Notice that for $\eta=0.05$, due to a lower boundary value, the control inputs $u^*(x,t_0)$ push the robot towards the inner or outer boundary from the most part of $\mathcal{X}_{s}$, except in the neighborhood of the origin. Whereas, for $\eta=0.13$, aggressive inputs $u^*(x,t_0)$ are applied near the boundary to force the robot to go towards the origin. In Fig. \ref{Fig. sample trajs}, we have plotted $100$ sample trajectories generated using synthesized optimal policies for four different values of $\eta$. The trajectories are color-coded; the black paths collide with the boundaries, while the blue paths converge in the neighborhood of the origin. We observe that as $\eta$ increases, the weight of blue paths increases whereas that of black paths reduces. In other words, at the lower values of $\eta$, there is a greater chance of the robot going out of $\mathcal{X}_{s}$ as compared to the higher values of $\eta$. We record the average computation times of FDM and path integral for running an optimal trajectory from the given start position to the origin for $\eta=0.13$. Both the algorithms are implemented in MATLAB on a consumer laptop. The average computation time for path integral was $51.35$ seconds, while for FDM, the off-line computation took $9.57$ seconds and the on-line table look-up took $0.17$ seconds. These numbers are conservative since no effort was made to optimize the algorithms for speed. The recorded computation times imply that for this particular example, FDM is more efficient than path integral. However, as mentioned in Section \ref{Sec: FDM}, the time complexity of FDM is expected to grow exponentially as the state-space dimension increases. In future work, we plan to study the scalability of FDM and path integral methods for solving the HJB PDE (\ref{HJB PDE}).\par

Table \ref{Table: P_fail} compares the failure probabilities of the policies synthesized by FDM and path integral as $\eta$ increases. Failure probabilities are computed using two approaches; first using the PDE (\ref{risk PDE}) $P_{fail}^{(1)}$, and second using the na\"ive Monte Carlo sampling $P_{fail}^{(2)}$. The second and third columns of Table \ref{Table: P_fail} list the failure probabilities of the policies synthesized by FDM. The values in the second column are obtained by solving the PDE (\ref{risk PDE}) also via FDM, whereas that in the third column are obtained by na\"ive Monte Carlo sampling. The last column lists the failure probabilities of the policies synthesized by path integral. Both $P_{fail}^{(1)}$ and $P_{fail}^{(2)}$ are the same for the path integral, since it inherently uses Monte Carlo sampling technique. For na\"ive Monte Carlo, $100$ sample trajectories are used in both FDM and path integral. 
\begin{table}[t]
\caption{Failure probabilities of the policies synthesized by FDM and path integral as $\eta$ increases. $P_{fail}^{(1)}$ is computed via the PDE (\ref{risk PDE}), while $P_{fail}^{(2)}$ is computed via the na\"ive Monte Carlo sampling with $100$ sample trajectories.}
\label{Table: P_fail}
\begin{center}
\begin{tabular}{ |c| c | c|c|} 
 \hline
 \multirow{2}{*}{$\eta$} & \multicolumn{2}{c|}{FDM} &  Path Integral\\
\cline{2-4}
  & $P_{fail}^{(1)}$ & $P_{fail}^{(2)}$  & $P_{fail}^{(1,2)}$ \\ 
 \hline
 $0.05$& 0.9987 &1& 0.99  \\

 $0.07$& 0.9908 & 0.97& 0.98 \\

 $0.09$& 0.9363 &0.93& 0.93 \\

$0.10$& 0.8442 &0.83& 0.82 \\

 $0.11$& 0.6672 &0.70& 0.65 \\

 $0.12$& 0.4292 &0.47& 0.43\\

 $0.13$& 0.2264 &0.24& 0.29 \\
\hline
\end{tabular}
\end{center}
\end{table}
\begin{figure}[t]
    \centering
\includegraphics[scale=0.17]{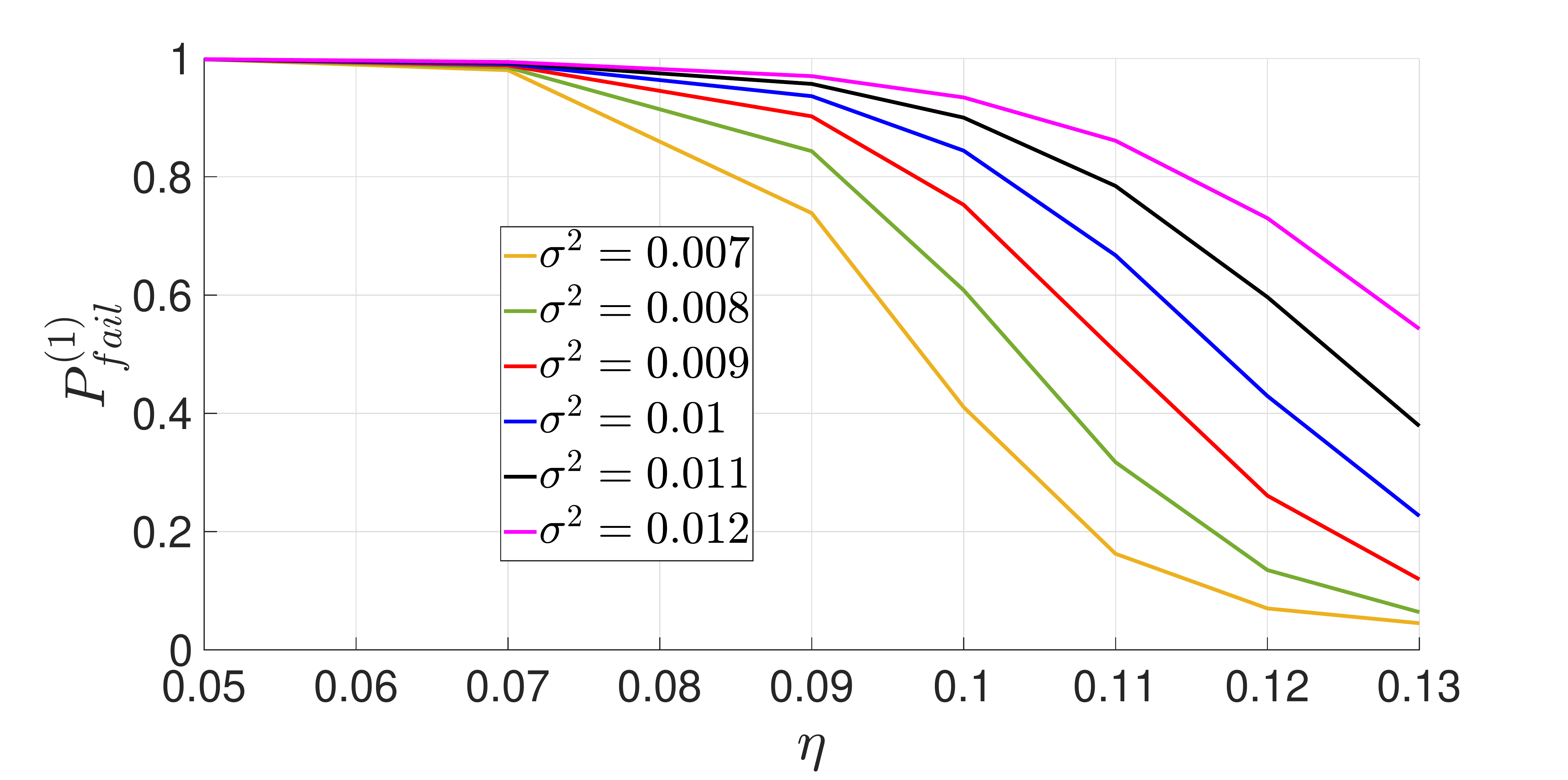} 
        \caption{Failure probabilities of the optimal policies (synthesized by FDM) for different values of $\sigma^2$ as a function $\eta$. The values are computed by solving the PDE (\ref{risk PDE}) via FDM.} 
        \label{Fig. pfail with s2}
\end{figure}
Fig. \ref{Fig. pfail with s2} shows how the failure probabilities change with $\eta$ for different noise levels ($\sigma^2$). These values are computed by solving the PDE (\ref{risk PDE}) via FDM for the optimal policies synthesized also via FDM. The simulation results presented in Fig. \ref{Fig. surf plots FDM}-\ref{Fig. pfail with s2} and Table \ref{Table: P_fail} demonstrate how $\eta$ can be tuned appropriately to achieve a desired level of safety. \par

\section{Conclusion}
The paper presented an HJB-PDE-based solution approach for a risk-constrained SOC problem. We formulated a risk-constrained SOC problem using the notion of exit time and converted it to a risk-minimizing control problem that has a time-additive cost function. It is shown that the risk-minimizing control synthesis is equivalent to solving an HJB PDE whose boundary condition can be tuned appropriately to achieve a desired level of safety. We also showed that the proposed risk-minimizing control problem can be viewed as a generalization of the risk estimation problem that can be solved via a PDE which is a special case of the HJB PDE. The finite difference and path integral methods are applied to solve a class of risk-minimizing SOC problems whose associated HJB equation is linearizable via the Cole-Hopf transformation. The proposed control synthesis framework is validated using a 2D robot navigation problem and the results obtained using FDM and path integral are compared. The analysis in this paper assumed the Lagrange multiplier $\eta$ to be a given constant. The future work will focus on establishing the correspondence between $\eta$ and the user-specified risk bound $\Delta$ and comparing our approach against the state-of-the-art chance-constrained optimal control approaches. We also plan to validate our framework on more complex and higher dimensional dynamics, and study the scalability of FDM and path integral methods.  

\bibliographystyle{IEEEtran}
\bibliography{bibliography}

\end{document}